\documentclass[11pt,a4paper]{article}

\usepackage{times}
\usepackage[T1]{fontenc}
\usepackage[utf8]{inputenc}
\usepackage[textwidth=0.7\paperwidth,textheight=0.75\paperheight,centering]{geometry}
\usepackage{float}

\usepackage{amsmath}
\usepackage{amsfonts}
\usepackage{amssymb}
\usepackage{euscript}
\usepackage{mathrsfs}
\usepackage{bbm}

\usepackage{graphicx}
\usepackage{tikz}
\usetikzlibrary{positioning,arrows,plotmarks,decorations.pathreplacing,shapes}

\usepackage{boxedminipage}
\usepackage{algpseudocode}

\usepackage{amsthm}

\theoremstyle{plain}

\newtheorem{proposition}{Proposition}
\newtheorem{lemma}{Lemma}

\theoremstyle{definition}
\newtheorem{definition}{Definition}

\usepackage{natbib} 

\usepackage{enumerate}
\usepackage{url}
\usepackage{hyperref}

\def\RR {{\mathbbm R}}

\def\NN {{\mathbbm N}}

\def\btheta {\boldsymbol{\theta}}

\def\Uni {{\sf U}}

\def\Stb {{\sf S}}

\def\EE {{\mathbbm E}}
\def\PP {{\mathbbm P}}
\def\Var{{\mathbbm{V}ar}}

\def\isindist{\overset{\mathcal{D}}{=}}
\def\hasdist{\overset{\mathcal{D}}{\sim}}

\def\iid{\overset{\mathrm{iid}}{\sim}}
\newcommand{\indicator}[1]{\mathbbm{1}\{#1\}}
\def\sign{\mathrm{sign}}

\begin{document}

\title{Monte Carlo Methods for Insurance Risk Computation}

\author{
\parbox{6cm}{
Shaul Bar-Lev\\ 
{\small University of Haifa}\\
{\small Haifa, Israel}\\
{\small\url{barlev@stat.haifa.ac.il}}
}
\and
\parbox{6cm}{
Ad Ridder\\ 
{\small Vrije University}\\
{\small Amsterdam, Netherlands}\\
{\small\url{ad.ridder@vu.nl}}
}
}

\maketitle

\begin{abstract}
In this paper we consider the problem of computing tail probabilities 
of the distribution of a random sum of positive random variables. 
We assume that the individual variables follow a reproducible natural 
exponential family (NEF) distribution, and that the random number has 
a NEF counting distribution with a cubic variance function. 
This specific modelling is supported by data of the aggregated claim 
distribution of an insurance company. Large tail probabilities are 
important as they reflect the risk of large losses, however, 
analytic or numerical expressions are not available. 
We propose several simulation algorithms which are based on an 
asymptotic analysis of the distribution of the counting variable and 
on the reproducibility property of the claim distribution. 
The aggregated sum is simulated efficiently by importance
sampling using an exponential cahnge of measure.
We conclude by numerical experiments of these algorithms.
\end{abstract}

\section{Introduction}\label{s:intro}
Let $Y_1,Y_2,\ldots$ be i.i.d.\ positive 
random variables representing the individual
claims at an insurance company,
and let $N\in\NN_0$
designate the total number of claims occurring during a certain
time period.
The total amount of these claims is 
called the aggregated claim variable, 
denoteb by
\[
S_{N}=\sum_{k=1}^{N}Y_{k}.
\]
A major issue for insurance companies is the uncertainty 
of the occurrence of a large aggregated claim, because,
if this happens, the company faces large losses
that may ultimately lead to a ruin.
Thus, an important quantity to compute is the insurance risk
factor 

\begin{equation}\label{e:aggregatetail}
\ell(x)= \PP\big(\sum\nolimits_{k=1}^{N} Y_{k}>x\big),
\end{equation}
for large levels $x$. 
Because of its importance for insurance companies,
many actuarial studies deal with this problem,
see the monograph of \citet{kaas08}.
However, there are many other practical situations 
in which the object of interest is a random sum
of i.i.d.\ random variables \citep{bahnemann15}.
For instance, 
$S_N$ might represent the total loss of a financial institute
due to defaults of $N$ obligators with credit sizes 
$Y_1,Y_2,\ldots$.

For doing the actual computions of the risk factor,
one needs to fit a model for 
the counting distribution of $N$
and the claim size distribution of $Y$.
Nowadays we see that the Poisson and the Gamma distributions,
respectively, are often being used \citep{bowers97}.
Other proposals include negative binomial for the counting number 
and inverse Gaussian for the claim size.
 
However, due to large uncertainties, many realistic data 
show large overdispersion.
In fact, our study is motivated by available data of a 
car insurance company for which the traditional 
distributions clearly do not fit properly. 
The (empirical) variance of the counting number data shows 
a power law with respect to the (empirical) mean, with a power 
close to three. 
This observation was the reason that 
we decided to consider counting distributions with
tails that go beyond (are heavier than) the Poisson and 
negative binomial. A natural modeling technique to
introduce families of distributions is by 
considering the concept of natural exponential families 
\citep{dunn05,dunn08,letac90,smyth02}. In our case
we are interested in natural exponential families
with cubic variance functions \citep{letac90}. 
Concerning the counting variable $N$,
we shall investigate 
\begin{itemize}
\item
the Abel distribution;
\item
the strict arcsine distribution;
\item
the  Takacs distribution.
\end{itemize}
These are new distributions for insurance modelling, 
and have to our knowledge not been considered before
in computation and simulation studies. As said above,
our objective is to execute numerical computations
of the insurance risk factor, for which we
consider using Monte Carlo simulations,
the main reason being 
that there are no 
analytic expressions available.
Thus, a main part of our paper deals with developing 
the simulation algorithms
for generating samples from these distributions.

\bigskip\noindent
Also concerning the claim size distributions, we 
propose modelling by natural exponential families. 
Specifically, we consider
\begin{itemize}
\item
gamma distribution;
\item
positive stable distributions;
\item
inverse Gaussian distribution.
\end{itemize}
These are well-known distributions in insurance modeling,
for wich simulation algorithms for generating samples 
have been established 
\citep{chambers76,devroye86,michael76,shuster68}.

\bigskip\noindent
In this way, our aggregate models become  
Tweedie models in the sense that both the
distributions of the counting number and the distributions
of the claim size belong to natural exponential families 
\citep{dunn05,dunn08,smyth02}.
Hence, we shall investigate whether the statistical procedures for estimating the parameters in these models can be applied to our data, 
or whether we need to develop other procedures.  
Commonly one models the mean and dispersion 
in terms of risk factors, for instance by regression models or 
by generalized linear models \citep{smyth02}.
However, we propose to directly compute the risk as a tail probability 
of the aggregated claim distribution by executing
Monte Carlo simulations. 

\bigskip\noindent
The simulation algorithm exploits two efficiency improvements
with respect to standard Monte Carlo.
Firstly, the claim size distributions show the reproducibility
property \citep{barlev86}, which says that convolutions
can be considered being transformations of univariates.
Thus, for example, 
a single sample of the inverse Gaussian distribution suffices
for generating a sum of i.i.d.\ inverse Gaussians.
Secondly, we apply importance sampling by implementing a change of
measure which is based on the well-known exponentially tilting
the probability distributions \citep[Chapter VI]{asmussenglynn07}. 
The optimal tilting factor 
is determined by a saddle-point equation, and results in 
a logarithmically efficient estimator.

\bigskip\noindent
The paper is organized as follows. Section \ref{s:nef}
summarizes the concepts of Tweedie NEF distributions, and 
reproducibility. The main contribution of the paper 
is contained in Section \ref{s:counting} where we
analyse the three counting distributions which leads to
the construction of the simulation algorithms for generating
samples. Section \ref{s:claim} summarizes a few 
aspects of the claim distributions. 
The aggregated claim risks  are computed in Section \ref{s:computing}
by Monte carlo simulation using the algorithms that
we have developed. We show how these risks for large levels 
can be computed afficiently  by an appropriate change
of measure for importance sampling. Finally, Section \ref{s:case}
gives details of the data that motivated this work.


\section{Natural Exponential Family and Reproducibility}\label{s:nef}
We summarize some concepts and properties 
of distributions from natural exponential families (NEF),
see \citet{dunn05,dunn08,letac90,smyth02}.

\begin{definition}
Let $\nu$ be a non-Dirac positive Radon measure on $\RR$,
and $L(\theta)=\int e^{\theta x}\,\nu(dx)$ its Laplace transform.
Assuming that 
$\mathtt{int}\,\Theta = \mathtt{int}
\{\theta : L(\theta)<\infty\}\neq\emptyset$,
then the NEF generated by $\nu$ is define by
the probability distributions
\begin{equation}\label{e:nefdef}
\mathcal{F}=\left\{  F_{\theta} : 
F_{\theta}(dx)=e^{\theta x-\kappa(\theta)}\,\nu(dx),
\text{ }\theta\in\Theta\right\},
\end{equation}
where
$\kappa(\theta)=\log L(\theta)$, the cumulant transform of $\nu$,
is real analytic on $\mathtt{int}\,\Theta$.
\end{definition}

\noindent
The generating measure $\nu$ is called also the kernel of the
NEF. We may associate a random variable $X_\theta$ with the
NEF distribution $F_\theta$. Then
\[
\EE[X_\theta] = \kappa'(\theta);\quad
\Var[X_\theta] = \kappa''(\theta).
\]
Note that $\kappa'$ is invertibe, thus we obtain the 
NEF parameter $\theta$ by
\[
\theta=\theta(m) = \big(\kappa'\big)^{-1}(m).
\]
This means that if we let $\kappa(m)=\kappa\big(\theta(m)\big)$, 
we can represent
the NEF equivalently by
\[
\mathcal{F}=\left\{  F_{m} : F_{m}(dx)=e^{\theta(m)
x-\kappa(m)}\,\nu(dx),\text{ }m \in \mathcal{M}\right\},
\]
where $\mathcal{M}$ is the mean domain of $\mathcal{F}$
and is given by $\mathcal{M}=\kappa'(\mathtt{int}\,\Theta)$. 
Finally, if also the variance
$V(m)$ of a NEF distribution is given as function of the mean $m$,
the pair $(V,\mathcal{M})$ uniquely determines a
NEF within the class of NEF's.  
Simple algebra shows that $\theta$ and $\kappa(\theta)$
can be represented in terms of $m$ by

\begin{equation}\label{e:thetakappa}
\begin{split}
\theta_A(m) &= \int \theta'(m)\,dm = \int \frac{1}{V(m)}\,dm+A,\\
\kappa_B(m) &= \int \kappa'(m)\,dm = \int \frac{m}{V(m)}\,dm+B,
\end{split}
\end{equation}
where $A$ and $B$ are constants. Note that these functions are
not unique at this stage. 
The constants $A$ and $B$ need to be chosen such that 
the corresponding $F_\theta$ function is a probability distribution.

\bigskip\noindent
We call the NEF a Tweedie NEF when $V(m)=O(m^r)$,
i.e. a power law \citep{barlev86,jorgensen87,tweedie84}.
Furthermore, the following reproducibility concept 
will be a key element in
our analysis. It has been developed in \citep{barlev86}.

\begin{definition}\label{d:repro}
Let $\mathcal{F}$ be a NEF as in \eqref{e:nefdef}, and 
suppose that $X_1,X_2,\ldots\iid F_\theta$.
Denote $S_n=\sum_{k=1}^n X_k$.
The NEF is said to be reproducible  if there exist
a sequence of real numbers $(c_{n})_{n\geq 1}$, 
and a sequence of mappings $\{g_{n}:\Theta \to \Theta, n\geq 1\}$,
such that 
for all $n\in\NN$ and for all $\theta\in\Theta$

\[
c_{n}S_{n} \hasdist F_{g_{n}(\theta)}\in\mathcal{F}.
\]
\end{definition}

\section{Counting Distributions}\label{s:counting}
In this section we analyse discrete counting NEF's that are given
by a cubic VF (variance function), see \citet{letac90}.
As said in the Introduction, we are motivated by data in a case study
having a variance showing indeed such a power law.
Our distributions will be used for computing the insurance 
risk factor by simulations, and, thus, the issue is
how to generate samples from these distributions.
Our analysis will lead to the construction
of sampling algorithms that are based on the acceptance-rejection
method. As dominating proposal distribution
we can use the same distribution that is used 
to sample from the Zipf distribution \citep{devroye86}.

We analyse the Abel, the arcsine, and the Takacs NEF's,
consecutively. For each NEF we introduce the VF, develop relevant
asymptotics, and then propose our simulatio procedure.


\subsection{Abel NEF}
The VF is given by
\begin{equation}\label{e:abelvar}
V(m) = m\Big(1+\frac{m}{p}\Big)^{2},\quad m>0,p>0.
\end{equation}
By \eqref{e:thetakappa} we deduce that the NEF-parameter function
$\theta_A(m)$ and the cumulant function $\kappa_B(m)$ 
as functions of mean $m$ are derived by

\begin{align}
\theta_A(m)&=\int \frac{1}{V(m)}\,dm=
\int \frac{1}{m\big(1+\frac{m}{p}\big)^{2}}\,dm=
\log \frac{m}{m+p}+\frac{p}{m+p}+ A (\mbox{a constant}) \label{e:abeltheta}\\
\kappa_B(m)&=\int \frac{m}{V(m)}\,dm=
\int \frac{1}{\big(1+\frac{m}{p}\big)^{2}}\,dm=
-\frac{p^{2}}{m+p}+B (\mbox{a constant}). \label{e:abelkappa}
\end{align}

\noindent
The kernel is given in \citep{letac90}:

\begin{equation}\label{e:abelkernel}
\nu(n)=\frac{1}{n!}p(p+n)^{n-1},\quad n\in\NN_0.
\end{equation}

\begin{proposition}
$A=-1$  and $B=p$.
\end{proposition}

\begin{proof}
The constant $B$ follows from Proposition 4.4 of [Letac \& Mora, 1990]:

\[
\nu(0)=e^{\kappa(0)}\;\Leftrightarrow\;\;
1 = e^{-p+B}\;\Leftrightarrow\;\; B=p.
\]

\noindent
Hence,

\begin{equation}\label{e:abelkappa}
\kappa(m) = -\frac{p^{2}}{m+p}+p=\frac{mp}{m+p}\;\;
\Rightarrow\;\;\frac{\kappa(m)}{p}= \frac{m}{m+p}. 
\end{equation}

\noindent
The constant $A$ follows from detailed readings of [Letac \& Mora, 1990].
Specifically, Theorem 4.5. The Abel distribution follows by defining
the generating function

\[
g(z) = e^z = \sum_{n=0}^\infty \frac{1}{n!}z^n,
\] 
and 
\[
\nu(n) = \frac{p}{p+n}\,\frac{1}{n!}\,\Big(\frac{d}{dz}\Big)^n 
\, g^{n+p}(z)\Big|_{z=0}.
\]

\noindent
Indeed, you get \eqref{e:abelkernel}. Furthermore, display (4.27) in
[Letac \& Mora, 1990] says

\[
e^{\theta(m) + \kappa(m)/p} = g^{-1}\Big(e^{\kappa(m)/p}\Big).
\]

\noindent
Substituting $g^{-1}=\log$ and the expression for $\kappa(m)/p$
in display \eqref{e:abelkappa}, we get:

\[
e^{\theta(m) + m/(m+p)} = m/(m+p)
\;\;\Leftrightarrow\;\; 
\theta(m) = \log\frac{m}{m+p} - \frac{m}{m+p}=
\log\frac{m}{m+p} + \frac{p}{m+p} -1.
\]

\noindent
Conclusion $A=-1$.
\end{proof}

\medskip\noindent
Define 
\begin{align*}
\nu_0(n) &= \nu(n)e^{-n-p}=\frac{pe^{-p}}{n!}(p+n)^{n-1}e^{-n}\\
\theta(m) &= \theta_A(m)-A=\log \frac{m}{m+p}+\frac{p}{m+p}\\
\kappa(m) &=\kappa_B(m)-B=-\frac{p^2}{m+p}
\end{align*}

\noindent
Then, the NEF Abel counting probability mass function (pmf)
of the associated counting variable $N_\theta$ is:
\begin{equation}\label{e:abelpmf}
\PP(N_\theta=n)=f_\theta(n)=\nu_0(n)\, e^{n\theta(m)-k(m)}.
\end{equation}

\noindent
Conveniently we omit the NEF parameter $\theta$ in our notations
when there is no confusion.

\subsubsection{Analysis}
First we consider an asymptotics of the modeified kernel $\nu_0(n)$,
using the Stirling approximation:
\[
n! \sim \sqrt{2\pi n}\,\big(n e^{-1}\big)^n,\quad n\to\infty,
\]
where $\sim$ means that the ratio converges to 1 (for $n\to\infty$).
This gives
\[
\nu_0(n)\sim \frac{pe^{-p}(p+n)^{n-1}\,e^{-n}}%
{\sqrt{2\pi n}\,\big(n e^{-1}\big)^n}
=\frac{pe^{-p}}{\sqrt{2\pi}}\, \frac{1}{n\sqrt{n}}\,\Big(1+\frac{p}{n}\Big)^{n-1}
\sim \frac{p}{\sqrt{2\pi}}\, \frac{1}{n\sqrt{n}}.
\]
The right-hand side shows correspondence with a
Zipf distribution \citep{devroye86}:
\[
z(n)=\frac{1}{\zeta(3/2)}\, \frac{1}{n\sqrt{n}}, 
\quad n=1,2,\ldots,
\]
where $\zeta(\cdot)$ is the Riemann zeta function. Sampling
from this Zipf distribution is done by an
accept-reject algorithm using the dominating
pmf $b(n)$ of the random variable $\lfloor U^{-2}\rfloor$:
\begin{equation}\label{e:zipfdominator}
b(n)=\frac{1}{\sqrt{n+1}}\,\Big(\sqrt{1+\tfrac1n} - 1\Big),
\quad n\in\NN.
\end{equation}
The multiplication factor for $z(n)\leq cb(n)$ is
\citep{devroye86}
\[
c = \frac{\sqrt{2}}{\zeta(3/2)\,\big(\sqrt{2}-1\big)}.
\]
We show that we can use $b(n)$ also as proposal dominating
pmf for our NEF pmf $f(n)$. However, because the domains of
$b(n)$ and $f(n)$ differ ($\NN$ versus $\NN_0$),
we need a minor tweak. Denote the conditional pmf by
$f(n|n\geq 1) = \PP(N=n|N\geq 1)$.

\begin{lemma}\label{l:abeldominating}
There is a constant $C$ (not dependent on $n$) such that
\[
f(n|n\geq 1)\leq Cb(n),\quad n=1,2,\ldots.
\]
\end{lemma}
\begin{proof}
Note that $f(n|n\geq 1)=f(n)/\big(1-f(0)\big)$, with
$f(0)=\nu_0(0)e^{-\kappa(m)}=e^{-p-\kappa(m)}$.
Furthermore, for $n\geq 1$ using the lower bound 
$n!\geq \sqrt{2\pi n} \, n^{n}e^{-n}$, we get:
\begin{align*}
& \frac{(p+n)^{n-1}\,e^{-n}}{n!}
\leq \frac{(p+n)^{n-1}}{\sqrt{2\pi}\,n\sqrt{n}\, n^{n-1}}\\
&=\frac{1}{\sqrt{2\pi}\,n\sqrt{n}}\, \big(1+\frac{p}{n}\big)^{n-1}
\leq \frac{1}{\sqrt{2\pi}\,n\sqrt{n}}\,e^p.
\end{align*}
Moreover, clearly 
\[
\theta(m)=\log \frac{m}{m+p}+\frac{p}{m+p}\leq 0.
\]
All together,
\begin{align*}
& f(n|n\geq 1) = \frac{pe^{-p}}{(1-f(0))n!}(p+n)^{n-1}e^{-n}\,
e^{n\theta(m)-k(m)}\\
&\leq \frac{pe^{-\kappa(m)}}{(1-f(0))\sqrt{2\pi}\,n\sqrt{n}}
=\frac{\zeta(3/2)pe^{-\kappa(m)}}{(1-f(0))\sqrt{2\pi}}\, z(n)\\
&\leq \frac{\zeta(3/2)pe^{-\kappa(m)}}{(1-f(0))\sqrt{2\pi}}\, 
\frac{\sqrt{2}}{\zeta(3/2)\,\big(\sqrt{2}-1\big)}\,b(n)
=\underbrace{\frac{pe^{-\kappa(m)}}{(1-f(0))\sqrt{\pi}\,
\big(\sqrt{2}-1\big)}}_{=C}\,b(n),
\end{align*}
where $e^{-\kappa(m)} = e^{p^2/(p+m)}$.
\end{proof}

\medskip\noindent
\textbf{Remark.}
Note that when $m\gg p$, the constant $C$ is of order $p$ which
is reflected in the acceptance probability $1/C$ in the
accept-reject sampling algorithm. However,
for large dispersion parameters $p$ the larger constant $C$ deteriorates 
this algorithm. In that case one might improve 
bounding the kernel and the probabilities. 
Our case study gave $m\gg p$, so we decided to implement the bounding 
as given above. That gave acceptance probability $0.25$.

\medskip\noindent
Summarizing, the Monte Carlo algorithm for simulating from the Abel
distribution \eqref{e:abelpmf} becomes:

\begin{center}
\begin{boxedminipage}{13cm}
\textsc{Acceptance-Rejection Algorithm for NEF Abel Distribution}
\begin{algorithmic}[1]
\State Generate $U\hasdist\Uni(0,1)$.
\If{$U<f(0)$}
\State $N\gets 0$.
\Else
\Repeat
\State Generate $U\hasdist\Uni(0,1)$.
\State Set $N=\lfloor U^{-2}\rfloor$.
\State Compute $P=\frac{f(N|N\geq 1)}{C\,b(N)}$.
\State Generate $U\hasdist\Uni(0,1)$.
\Until{$U < P$}
\EndIf
\State \textbf{return} $N$.
\end{algorithmic}
\end{boxedminipage}
\end{center}


\subsection{Arcsine NEF}
The VF is given by
\begin{equation}\label{e:arcsinvar}
V(m) = m\Big(1+\frac{m^2}{p^2}\Big)=\frac{m}{p^2}(m^2+p^2),\quad m>0,p>0.
\end{equation}
By \eqref{e:thetakappa} we deduce that the NEF-parameter function
$\theta_A(m)$ and the log-moment generating function $\kappa_B(m)$ are derived by

\begin{align}
\theta_A(m)&=\int \frac{1}{V(m)}\,dm=
\int \frac{p^2}{m(m^2+p^2)}\,dm=
\ln m-\frac12\log(m^2+p^2)+ A (\mbox{a constant})\nonumber \\
& = -\frac12\log(1+(p^2/m^2))+ A (\mbox{a constant}) \label{e:arcsintheta}\\
\kappa_B(m)&=\int \frac{m}{V(m)}\,dm=
\int \frac{1}{\big(1+\frac{m^2}{p^2}\big)}\,dm=
p \,\arctan(m/p)+B (\mbox{a constant}). \label{e:arcsinkappa}
\end{align}

\noindent
The kernel is given in \citep{letac90}:

\begin{align*}
\nu(2n)&=\frac{1}{(2n)!}\prod_{i=0}^{n-1}\big((2i)^2+p^2\big)\\
\nu(2n+1)&=\frac{p}{(2n+1)!}\prod_{i=0}^{n-1}\big((2i+1)^2+p^2\big).
\end{align*}

\begin{proposition}
$A=0$  and $B=0$.
\end{proposition}

\begin{proof}
The constant $B$ follows from Proposition 4.4 of [Letac \& Mora, 1990]:

\[
\nu(0)=e^{\kappa(0)}\;\Leftrightarrow\;\;
1 = e^{B}\;\Leftrightarrow\;\; B=0.
\]

\noindent
Hence,

\begin{equation}\label{e:arcsinekappa}
\kappa(m) = p \,\arctan(m/p)\;\;
\Rightarrow\;\;\frac{\kappa(m)}{p}= \arctan(m/p). 
\end{equation}

\noindent
The generating function of the arcsine kernel is
(see Example C in [Letac \& Mora, 1990])
\[
f(z)=\sum_{n=0}^\infty\nu(n)z^n = e^{p\arcsin z}.
\]

\noindent
Because, $\kappa(\theta)=\log f(e^\theta)$, and 
$\kappa(m)=\kappa\big(\theta(m)\big)$, we get

\begin{align*}
&\kappa(m)=\log e^{p\arcsin e^{\theta(m)}}
=p\arcsin e^{\theta(m)}\\
&=p\arcsin e^{-\log\sqrt{1+(p^2/m^2)}+ A }
=p\arcsin \frac{e^A}{\sqrt{1+(p^2/m^2)}}\\
& \Rightarrow\;\; \sin \frac{\kappa(m)}{p}=\frac{e^A}{\sqrt{1+(p^2/m^2)}}.
\end{align*}

\noindent
According to display \eqref{e:arcsinekappa}:
\[
\sin \frac{\kappa(m)}{p}=\sin \arctan(m/p)=
\frac{m/p}{\sqrt{1+(m^2/p^2)}},
\]
the last equation a well-known identity of trigonometric functions.
Equating:
\[
\frac{e^A}{\sqrt{1+(p^2/m^2)}}=\frac{m/p}{\sqrt{1+(m^2/p^2)}}
\;\;\Leftrightarrow\;\; e^A=1.
\]

\noindent
Conclusion $A=0$.
\end{proof}

\noindent
Denote $\theta(m)=\theta_0(m)$, and $\kappa(\theta)=\kappa_0(m)$.
Hence, we get the NEF Arcsine counting probability mass function (pmf)
of the counting variable $N$ (omitting NEF parameter $\theta$
in the index notation):
\begin{equation}\label{e:arcsinepmf}
\PP(N=n)=f(n)=\nu(n)\, e^{n\theta(m)-k(m)},\quad  n\in\NN_0.
\end{equation}

\subsubsection{Analysis}
In the Appendix we show
that there is a constants $K$ such that for $n=1,2,\ldots$
\[
\nu(2n)\leq K\frac{1}{n\sqrt{n}}.
\]
Thus, for these even terms we recognize again
the Zipf distribution. This will be helpfull to
find a dominating proposal distribution.

\begin{lemma}
Define the double Zipf dominating distribution $b_2(n), n=2,3,\ldots$
by
\[
b_2(2n)=b_2(2n+1)=\frac12b([n/2]),
\]
where $b(n), n=1,2,\ldots$ is the pmf that dominates the Zipf pmf, 
defined in \eqref{e:zipfdominator}.
Then there is a constant $C$ such that
\[
f(n|n\geq 2)=\PP(N=n|N\geq 2)\leq Cb_2(n),\quad n=2,3,\ldots.
\]
\end{lemma}
\begin{proof}
The NEF parameter satisfies
\[
\theta(m) = -\frac12\log(1+(p^2/m^2))\leq 0.
\] 
Let $A = f(0)+f(1)=\PP(N=0)+\PP(N=1)$.
Because $f(n)=\nu(n)e^{\theta(m)n-\kappa(m)}$, we can
bound the probabilities $f(2n|n\geq 1)=\PP(N=2n|N\geq 2)$
by
\[
f(2n|n\geq 1) = \frac{f(2n)}{1-A} 
\leq \frac{Ke^{-\kappa(m)}}{1-A}\,\frac{1}{n\sqrt{n}} \leq  
\widetilde{C}\,b(n),
\]
where 
\[
\widetilde{C} = \frac{Ke^{-\kappa(m)}}{1-A}\,\frac{\sqrt{2}}{\sqrt{2}-1}.
\]
Then $f(2n|n\geq 1) \leq C b_2(2n)$
for $C=2\widetilde{C}$.

\medskip\noindent
The constant $K$ depends on a treshold $i^*$ such that 
for $n>i^*$, $f(2n+1)<f(2n)$, 
and for $n\leq i^*$,  $\nu(2n+1) \leq K n^{-3/2}$ (see the Appendix).
Thus also all odd terms satisfy $f(2n+1|n\geq 1) \leq C b_2(2n+1)$.
\end{proof}

\medskip\noindent
\textbf{Remark.}
Similarly to our algorithm for sampling from the Abel distribution, 
also the constant $C$ becomes larger for larger $p$, deteriorating 
the accept-reject sampling method. In our implementation we
included one more term in the bounding procedure that is described
in the Appendix. This gave an acceptance ratio of $0.34$.

\medskip\noindent
Summarizing, the Monte Carlo algorithm for simulating from the arcsine
distribution \eqref{e:arcsinepmf} becomes:

\begin{center}
\begin{boxedminipage}{13cm}
\textsc{Simulation Algorithm for NEF Arcsine Distribution}
\begin{algorithmic}[1]
\State Generate $U\hasdist\Uni(0,1)$.
\If{$U<f(0)$}
\State $N\gets 0$.
\Else
\If{$U<f(0+f(1)$}
\State $N\gets 1$.
\Else
\Repeat
\State Generate $U\hasdist\Uni(0,1)$.
\State Set $Y=\lfloor U^{-2}\rfloor$.
\State Generate $U\hasdist\Uni(0,1)$.
\If{$U<0.5$}
\State Set $N=2Y$.
\Else
\State Set $N=2Y+1$.
\EndIf
\State Compute $P=\frac{f(N|N\geq 2)}{C\,b_2(N)}$.
\State Generate $U\hasdist\Uni(0,1)$.
\Until{$U < P$}
\EndIf
\EndIf
\State \textbf{return} $N$.
\end{algorithmic}
\end{boxedminipage}
\end{center}


\subsection{Takacs}
The variance function is given by
\begin{equation}\label{e:takcsvar}
V(m) = m\Big(1+\frac{m}{p}\Big)\Big(1+\frac{2m}{p}\Big),
\quad m,p>0.
\end{equation}

\noindent
By \eqref{e:thetakappa} we deduce that the NEF-parameter function
$\theta_A(m)$ and the log-moment generating function $\kappa_B(m)$ are 
derived using partial-fraction decomposition:

\begin{align*}
\theta_A(m)&=\int \frac{1}{V(m)}\,dm=
\int \Big(\frac{1}{m} + \frac{1/p}{1+m/p} -
\frac{4/p}{1+2m/p}\Big)\,dm = 
\int \Big(\frac{1}{m} + \frac{1}{p+m} -
\frac{4}{p+2m}\Big)\,dm \\
&= \log m + \log(p+m) - 2\log(p+2m) + A
=\log\frac{m(p+m)}{(p+2m)^2} + A.\\
\kappa_B(m)&=\int \frac{m}{V(m)}\,dm=
\int \Big(\frac{-1}{1+m/p} + \frac{2}{1+2m/p}\Big)\,dm
=\int \Big(\frac{-p}{p+m} + \frac{2p}{p+2m}\Big)\,dm\\
&=-p\log(p+m) + p\log(p+2m) +B= p\log\frac{p+2m}{p+m} +B.
\end{align*}

\noindent
The kernel is given in \citep{letac90}:

\begin{equation}\label{e:takacskernel}
\nu(n)=\frac{p}{n+p}\,\frac{1}{n!}
\,(n+p)(n+p+1)\cdots(n+p+n-1),\quad n\in\NN_0.
\end{equation}

\begin{proposition}
$A=0$  and $B=0$.
\end{proposition}

\begin{proof}
The constant $B$ follows from Proposition 4.4 of [Letac \& Mora, 1990]:

\[
\nu(0)=e^{\kappa(0)}\;\Leftrightarrow\;\;
1 = e^{B}\;\Leftrightarrow\;\; B=0.
\]

\noindent
Hence,

\begin{equation}\label{e:takacskappa}
\kappa(m) = p\log\frac{p+2m}{p+m}\;\;
\Rightarrow\;\;\frac{\kappa(m)}{p}= \log\frac{p+2m}{p+m}. 
\end{equation}

\noindent
The constant $A$ follows from detailed readings of [Letac \& Mora, 1990].
Specifically, Theorem 4.5. The Takacs distribution follows by defining
the generating function

\[
g(z) = (1-z)^{-1} = \sum_{n=0}^\infty z^n,
\] 
and 
\[
\nu(n) = \frac{p}{p+n}\,\frac{1}{n!}\,\Big(\frac{d}{dz}\Big)^n 
\, g^{n+p}(z)\Big|_{z=0}.
\]

\noindent
Indeed, you get \eqref{e:takacskernel}. Furthermore, display (4.27) in
[Letac \& Mora, 1990] says

\[
e^{\theta(m) + \kappa(m)/p} = g^{-1}\Big(e^{\kappa(m)/p}\Big).
\]

\noindent
Substituting $g^{-1}(y)=(y-1)/y=1-(1/y)$ 
and the expression for $\kappa(m)/p$
in display \eqref{e:takacskappa}, we get:

\[
y = e^{\kappa(m)/p} = \frac{p+2m}{p+m}
\;\;\Rightarrow\;\;
g^{-1}\Big(e^{\kappa(m)/p}\Big)=1 - \frac{p+m}{p+2m}
=\frac{m}{p+2m},
\]

\noindent
thus,
\begin{align*}
& e^{\theta(m) + \kappa(m)/p} = g^{-1}\Big(e^{\kappa(m)/p}\Big)\\
& \Leftrightarrow\;\;
\theta(m) + \log\frac{p+2m}{p+m} = \log\frac{m}{p+2m}\\
&\Leftrightarrow\;\; 
\theta(m) = \log\frac{m}{p+2m} - \log\frac{p+2m}{p+m}=
\log\frac{m(p+m)}{(p+2m)^2}.
\end{align*}

\noindent
Conclusion $A=0$.
\end{proof}

\noindent
Define 
\[
\nu_0(n) = \nu(n)e^{\theta_0(m) n}; \;\;
\theta(m) = 0; \;\;
\kappa(m) =\kappa_0(m)
\]

\noindent
The NEF Takacs counting probability mass function (pmf)
of the random variable $N$ is
\begin{equation}\label{e:takacspmf}
\PP(N=n)=f(n)=\nu_0(n)\, e^{-k(m)}.
\end{equation}

\noindent
In the appendix we show that there is a constant $K$
such that
\[
\nu_0(n)\leq
K\,\frac{1}{n\sqrt{n}},
\quad n=1,2,\ldots.
\]

\begin{lemma}\label{l:takacsdominating}
There is a constant $C$ (not dependent on $n$) such that
\[
f(n|n\geq 1)\leq Cb(n),\quad n=1,2,\ldots.
\]
\end{lemma}
\begin{proof}
It follows immediately,
\begin{align*}
f&(n|n\geq 1) = \frac{f(n)}{1-f(0)}
=\nu_0(n)\frac{e^{-\kappa(m)}}{1-f(0)}
\leq \frac{Ke^{-\kappa(m)}}{1-f(0)}\frac{1}{n\sqrt{n}}\\
&= \frac{K\,e^{-\kappa(m)}\zeta(3/2)}{1-f(0)}\,z(n)
\leq  \frac{K\,e^{-\kappa(m)}\sqrt{2}}{(1-f(0))(\sqrt{2}-1)}\,b(n).
\end{align*}
\end{proof}

\noindent
The associated Monte Carlo algorithm for generating
Takacs samples is similar as the Abel algorithm.
The acceptance ratio in our case study is $0.23$.

\subsection{Concluding Remarks}
These three counting distributions have tails that are much fatter
than the more often used Poisson and negative binomial 
distributions. As an example, we consider in Section \ref{s:case}
a case study where the data show a mean $m\approx 70$, and
the variance $V\approx 52000$, which indicates a power function
$V(m)\approx m^r$ with $r\approx 2.5$. This was one of the reasons
to consider our specific counting distributions. 

An important feature of these distributions is their large tails.
Figure \ref{f:tails} shows the probability mass functions
for $1000\leq n\leq 1200$. The Poisson probabilities in this
region are virtually zero. The Abel and Takacs distributions
behave in the tails equivalently, while the arcsine shows
slightly lighter tails.

\begin{figure}[H]
\centering
\includegraphics[width=0.6\textwidth]{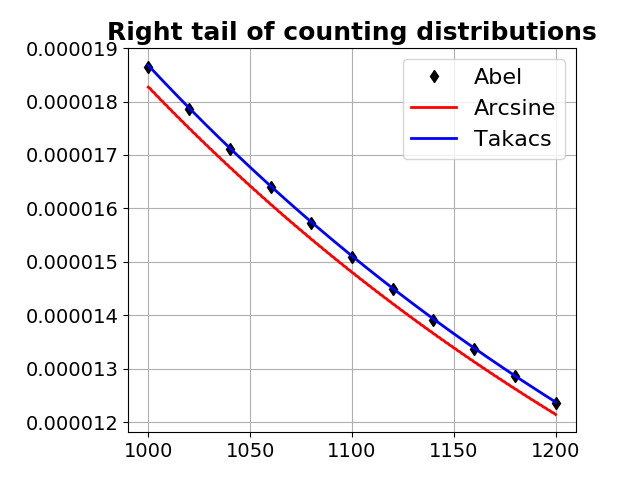}

\caption{Part of the probability mass functions.}%
\label{f:tails}%
\end{figure}


\section{NEF Claim Distributions}\label{s:claim}
For modeling the individual claim $Y$,
we consider positive reproducible NEF densities represented by
\[
f(y;\theta,p) = f(y)\,e^{\theta y-\kappa(\theta)},\quad y>0.
\]

\begin{itemize}
\item
Gamma given by kernel
\[
\nu(y) = \frac{y^{p-1}\, e^{-y}}{\Gamma(p)},\quad y>0,
\]
with dispersion parameter $p>0$.
The VF $V(m) = \frac{m^2}{p}$
yields by  Section \ref{s:nef}
\[
\theta(m) = 1-\frac{p}{m};\quad
\kappa(m) = p\,\log\frac{m}{p}.
\]
By inversion we get
\[
m(\theta)=\frac{p}{1-\theta}; \quad
\kappa(\theta)=p\,\log\frac{1}{1-\theta},
\]
for $\theta<1$. Hence, 
\[
f(y;\theta,p)
=\frac{(1-\theta)^p\,y^{p-1}\, e^{-(1-\theta)y}}{\Gamma(p)}.
\]
We observe that we  actually deal with a Gamma distribution with
shape parameter $p$ and scale parameter $1-\theta$,
and thus generating samples can be easily done \citep{devroye86}.

\medskip

Finally, 
let $Y_1,\ldots,Y_n\hasdist f(y;\theta,p)$ i.i.d., and
$S_n=\sum_{i=1}^nY_i$. Then 
$S_n\hasdist f(y;\theta,np)$.
This is not the same as reproducibility, but the NEF shows
the same property that the convolution can be represented
by a single distribution.

\item
Inverse Gaussian given by kernel
\[
\nu(y) = \frac{1}{\sqrt{2\pi p y^3}}\,
e^{\scalebox{1.0}{$- \frac{1}{2py}$}},\quad y>0,
\]
with dispersion parameter $p>0$.
The VF $V(m) = pm^3$
yields by  Section \ref{s:nef}
\[
\theta(m) = -\frac{1}{2pm^2};\quad
\kappa(m) = -\frac{1}{pm}.
\]
By inversion we get
\[
m(\theta) = \frac{1}{\sqrt{-2p\theta}};\quad
\kappa(\theta) = -\sqrt{\frac{-2\theta}{p}}, 
\]
for $\theta<0$. Hence, 
\begin{align*}
& f(y;\theta,p)
=\sqrt{\frac{1}{2\pi p y^3 }}\,
e^{\scalebox{1.0}{$- \frac{1}{2py}  + \theta y + \sqrt{\frac{-2\theta}{p}}$}}\\
& = \frac{1}{\sqrt{2\pi p y^3}}\,
e^{\scalebox{1.0}{$\frac{\theta}{y}\big(y+\sqrt{\frac{1}{-2p\theta}}\big)^2$}}.
\end{align*}
Setting 
\[
\delta = \frac{1}{\sqrt{p}};\quad
\gamma =\sqrt{-2\theta},
\]
we recognize the more traditional form of the Inverse Gaussian
pdf for which a simulation algorithm has been developed
\citep{michael76,shuster68}.

\medskip

Finally, 
let $Y_1,\ldots,Y_n\hasdist f(y;\theta,p)$ i.i.d., and
$S_n=\sum_{i=1}^nY_i$. Then 
\[
S_n\hasdist c_n f(c_n y; g_n(\theta),p),
\]
where 
\[
c_n = \frac{1}{n^2};\; g_n(\theta)=n^2\theta.
\]
See \citep{barlev86} for details.
Substituting these, we get after algebra,
\begin{equation}\label{e:igpdfsum}
S_n\hasdist f(y; \theta, p/n^2).
\end{equation}
\item
Positive $\alpha$-stable distribution.
Recall that a random variable $Y$
has an stable distribution with index $\alpha$,
denoted  
$Y\hasdist\Stb_\alpha(\sigma,\beta,\mu)$, if its
characteristic function satisfies (for convenience $\alpha\neq 1$):
\[
\log \phi(t) =  -\sigma^\alpha |t|^\alpha\,
\Big(1 - i\beta\sign(t)\,\tan\frac{\pi\alpha}{2}\Big) + i\mu t,
\]
for $t\in\RR$, where the parameters satisfy
\[
\alpha\in(0,2];\;\beta\in[-1,1];\; \mu\in\RR;\; \sigma>0,
\]
see e.g. \citet{nolan10,samorod94}. 
Since we consider positive variables $Y$, 
we get the so-called positive $\alpha$-stable distribution
by setting
$\alpha\in(0,1)$, $\beta=1$, $\mu\geq 0$. Furthermore we set
location parameter $\mu=0$ in which case
\[
\sigma = \Big(\cos \frac{\pi\alpha}{2}\Big)^{1/\alpha},
\]
and the cumulant generating function becomes \citep{feller71}
\[
\kappa(\theta) = -(-\theta)^\alpha,\quad\theta\leq 0.
\]
Both moments of the NEF-distributions $F(Y;\theta,p)$ are finite, whereas 
these are infinite for the kernel distribution
$F(y)$ which is positive $\alpha$-stable.

\medskip\noindent
Note that with this modeling the pdf $f(y),\,y>0$ is only
parameterized by index $\alpha$, but it is not given 
in explicit form. 
However, it generates a NEF with a power VF
\citep{barlev86,jorgensen87,tweedie84}
\[
V(m) = a m^p,
\]
where
\[
p = \frac{2-\alpha}{1-\alpha}>2;\; 
a=(1-\alpha)\alpha^{1/(\alpha-1)}>0.
\]
Also we obtain the $\theta$ and $\kappa$ function of mean $m$
and index $\alpha$:
\[
\theta(m) = -\Big(\frac{m}{\alpha}\Big)^{1/(\alpha-1)};\;\
\kappa(m) = -\Big(\frac{m}{\alpha}\Big)^{\alpha/(\alpha-1)}
\]
Thus, given mean $m$ and variance $V(m)$ we compute the
parameters $\theta$ and $\alpha$  for the NEF distribution
with pdf
\[
f(y;\theta,\alpha)=f(y)\,e^{\theta y-\kappa(\theta)},\;\;y>0.
\]
Generating samples from the NEF distribution is done by
acception-rejection algorithm, using $f(y)$ as proposal
pdf and $C=e^{-\kappa(\theta)}=e^{(-\theta)^\alpha}\geq 1$
as dominating factor. 
This follows directly from $\theta y\leq 0$. Furthermore,
generating from the proposal pdf $f(x)$ is based on
(i) generating from $\Stb_\alpha(1,1,0)$ distribution
by the Chambers algorithm \citep{chambers76}, and
(ii) the property $\Stb_\alpha(\sigma,1,0)\isindist 
\sigma \Stb_\alpha(1,1,0)$. 

\medskip

\begin{center}
\begin{boxedminipage}{13cm}
\textsc{Generating from $\theta$-NEF Positive $\alpha$-Stable Distribution}
\begin{algorithmic}[1]
\Repeat
\State Generate $X$ from $\Stb_\alpha(1,1,0)$.
\State $Y\gets \sigma X$.
\State Compute acceptance probability $P=e^{\theta Y}$.
\State Generate $U$ from uniform $(0,1)$ distribution.
\Until{$U < P$}
\State \textbf{return} $Y$.
\end{algorithmic}
\end{boxedminipage}
\end{center}

\noindent
Finally, positive $\alpha$-distributions satisfy
the reproducibility property \citep{barlev86}:
let $Y_1,\ldots,Y_n \hasdist f(y;\theta,\alpha)$ i.i.d., and
$S_n=\sum_{i=1}^nY_i$. Then 
\[
S_n\hasdist c_n f(c_n y; g_n(\theta),\alpha),
\]
where 
\[
c_n = n^{-1/\alpha};\;
g_n(\theta) = \frac{\theta}{c_n}=\theta\,n^{1/\alpha}.
\]
\end{itemize}


\section{Computing Insurance Risk}\label{s:computing}
The goal of our study is to compute efficiently the tail
probability $\ell=\PP(S_N>x)$ for large thresholds $x$, where
$S_N=\sum_{j=1}^NX_j$ is the random sum. We consider
Monte Carlo simulation while applying two
ideas: (i) the reproducibility, and (ii) importance sampling.
The reproducibility ensures that given $N=n$ has been generated or
observed, we generate $S$ as a single random variable
in stead of a sum (convolution). 

The standard Monte Carlo algorithm is trivial.
Let $M$ be the sample size, then the Monte Carlo estimator is
\[
\widehat{\ell}=\frac1M\sum_{i=1}^M\indicator{S^{(i)}_{N^{(i)}}>x},
\]
where in the $i$-th replication, the counting number
$N^{(i)}$ is generated from the counting distribution of interest
(Abel, arcsine, or Takacs), according to the algorithms of Section 
\ref{s:counting}. Given $N^{(i)}=n$, the aggregated claim size
$S^{(i)}_n$ is generated from the claim distribution
of interest (Gamma, inverse Gaussian, or positive $\alpha$-stable)
using the reproducibility property of Section \ref{s:nef}.
From the observations $\indicator{S^{(i)}_{N^{(i)}}>x}$, $i=1,\ldots,M$,
we compute the usual estimator and standard error (or confidence interval)
statistics.

However, if the threshold 
$x\gg \EE[S_N]=\EE[N]\EE[Y]$, we have difficulties in observing
the event $\{S_N>x\}$ when we apply the standard Monte Carlo
algorithm. As an illustration,
let $N$ be Abel and $Y$ be inverse Gaussian,
where the parameters are fitted by data in our case study 
of Section \ref{s:case}. The mean aggregate claim size 
$\EE[S_N]\approx 330$. Because our distributions have large tails, we consider
large levels $x$. As sample size we choose $M$ so large that the
standard error is about 10\% of the estimate. We see in
Table \ref{t:mcabelig} that the 
required sample sizes grow exponentially with level $x$
which means that very small probabilities are practically impossible
to compute.

\medskip

\begin{table}[H]
\caption{Estimates of $\PP(S_N>x)$ by Monte Carlo simulation
for Abel counting and inverse Gaussian claim
distributions.}
\label{t:mcabelig}
\medskip
\centering
\begin{tabular}{r|r|cc}
$x$ & $M$ & $\widehat{\ell}$ & std. error\\
\hline
5000 & 9000 & 1.08e-02 & 1.09e-03 \\
10000 & 37000 & 2.59e-03 & 2.64e-04 \\
15000 & 150000 & 6.47e-04 & 6.56e-05 \\
20000 & 410000 & 2.37e-04 & 2.40e-05 \\
25000 & 1020000 & 9.51e-05 & 9.66e-06
\end{tabular}
\end{table}

\subsection{Importance Sampling Algorithm} 
The idea of importance sampling is to
change the underlying probability measure of the stochastic system
in such a way that more samples are generated
from the target event. An unbiased estimator is
obtained by multiplying the observations with the likelihood ratio.
Denote the random variables that are generated 
in importance sampling by $\widetilde{N}$ and
$\widetilde{S}$, respectively. 
Suppose that $\widetilde{N}=n$, and
$\widetilde{S}=s$ are simulated, then the 
associated likelihood ratio is
\[
W(n,s)=\frac{\PP(N=n)}{\PP(\widetilde{N}=n)}\times
\frac{f_S(s)}{f_{\widetilde{S}}(s)}.
\]
The importance sampling estimator becomes
\[
\widehat{\ell}=\frac1M\sum_{i=1}^M
\indicator{\widetilde{S}^{(i)}_{\widetilde{N}^{(i)}}>x}
W(\widetilde{N}^{(i)},\widetilde{S}^{(i)}).
\]

\noindent
We have implemented the following importance sampling algorithm.
Let the parameters of the counting distribution be
$(\theta_N,p_N,m_N)$ (see Section \ref{s:counting}), 
and of the claim distribution
$(\theta_Y,p_Y,m_Y)$  (see Section \ref{s:claim}). 
These parameters are fitted to the data,
but note that the NEF-parameter $\theta$ follows from the 
mean-parameter $m$, and vice-versa, thus one of these suffices.
For the change of measure we propose changing the NEF-parameter 
(and consequently the mean-parameter), but not the dispersion parameter $p$.
In fact, we apply an exponential change of measure
using a common tilting parameter, say $\theta^*$,
for both the counting and the claim-size distribution.
This parameter is obtained as follows. Let
$\kappa(\theta) = \log \EE\big[\exp(\theta S_N)\big]$
be the cumulant generating function of the aggregated sum.
Then $\theta^*$ solves the saddlepoint equation $\kappa'(\theta)=x$;
thus
\[
=m_N(\theta_N+\theta^*) \times m_Y(\theta_Y+\theta^*)=x.
\]
The interpretation is that 
under the change of measure the most
likely samples of $S_N$ are generated around our target level $x$.
It is well-known in the rare-event theory that 
such a change of measure yields a logarithmically efficient
(or, asymptotically optimal) estimator in case of a fixed
number of light-tailed claims, i.e. $\PP(S_n>x)$, see
\citet[Chapter VI Section 2]{asmussenglynn07} or 
\citet[Chapter V Section 2]{bucklew04}.
However, by a conditioning argument one can show that the same
holds true for a random sum. This means that the required
sample sizes grow polynomially in level $x$,
which we can clearly see in Table \ref{t:isabelig}.
Our algorithm contains a minor tweak in that after
$\widetilde{N}$ has been generated, say $\widetilde{N}=n$,
we check whether $\EE[S_n]>x$. In that case, we generate $S_n$
from the original claim distribution, and otherwise we apply the change
of measure also for the claims.

\medskip

\begin{table}[H]
\caption{Estimates of $\PP(S_N>x)$ by importance sampling simulation
for Abel counting and inverse Gaussian claim
distributions.}
\label{t:isabelig}
\medskip
\centering
\begin{tabular}{r|r|cc}
$x$ & $M$ & $\widehat{\ell}$ & std. error \\
\hline
5000 &  4000 & 1.01e-02 & 9.09e-04\\
10000 &  6000 & 2.46e-03 & 2.43e-04\\
15000 & 10000 & 7.18e-04 & 6.88e-05\\
20000 & 14000 & 2.22e-04 & 2.23e-05\\
25000 & 16000 & 8.48e-05 & 8.40e-06\\
30000 & 20000 & 3.59e-05 & 3.65e-06\\
35000 &  26000 & 1.29e-05 & 1.25e-06\\
40000 &  34000 & 4.42e-06 & 4.41e-07 \\
45000 &  34000 & 2.18e-06 & 2.16e-07 \\
50000 &  40000 & 7.68e-07 & 7.78e-08
\end{tabular}
\end{table}


\section{Case Study}\label{s:case}
Data are available of claims at a car insurance company in Sweden
in a specific year\citep{hallin83,smyth11}.
The data consist of 2182 categories of 7 variables 
specifying per category:
kilometres, zone, bonus, make, insured, claims, payment.
Let $I$ be the set of categories, with $|I|$=2182. 
For any $i\in I$, we model the random variables
\begin{itemize}
\item
$N_i$: the number of claims in category $i$;
\item
$Y_i$: the claim size of a claimer in category $i$;
\item
$S_i$: the total amount of claims in category $i$.
\end{itemize}
The data give the numbers
$n_i$ of claimers, and $s_i$ of total claim amount,
they do not give the individual claim sizes.

\bigskip\noindent
For some subcatogories $J\subset I$ we propose that
the $N_j, \,j\in J$ are i.i.d. as $N$, and that the $Y_j, \, j\in J$
are i.i.d. as $Y$. Also we propose that $N$ and $Y$ are independent.
Data available are $n_j,\, j\in J$ observations
from $N$, and $s_j = \sum_{k=1}^{n_j}y_{jk}$ observations
from $S=\sum_{k=1}^N Y_k$ given $N$.

\bigskip\noindent
Let $\btheta_N$ be the vector of parameters of the
probability distribution of the counting variable $N$,
and $\btheta_Y$ of the claim size distribution
of $Y$. Due to the reproducibility property of $Y$,
the distribution of sum $S|(N=n)$ has the same parameter
vector $\btheta_Y$ (and the given number $n$).
For estimating these parameters we considered the two-moment
fit method because all our distributions are derived from
the mean and variance.

That is, 
let $\widehat{m}_{N}$ and  
$\widehat{v}_{N}$ be the sample average and variance
of the counting data $(n_j)_{j\in J}$. Then we fit a distribution
for $N$ such that
\[
\EE[N] = \widehat{m}_{N}\;\;\text{and}\;\;
\Var(N) = \widehat{v}_{N}.
\]
For the counting distributions of Section \ref{s:counting} we get

\begin{align*}
\mathrm{Abel}: & \quad
p = \frac{\widehat{m}_{N}\sqrt{\widehat{m}_{N}}}%
{\sqrt{\widehat{v}_{N}} - \sqrt{\widehat{m}_{N}}}\\
\mathrm{Arcsine}: & \quad
p = \frac{\widehat{m}_{N}\sqrt{\widehat{m}_{N}}}%
{\sqrt{\widehat{v}_{N} - \widehat{m}_{N}}}\\
\mathrm{Takacs}: & \quad
p = \frac{4\widehat{m}_{N}\sqrt{\widehat{m}_{N}}}%
{\sqrt{8\widehat{v}_{N} + \widehat{m}_{N}} - 3\sqrt{\widehat{m}_{N}} }
\end{align*}

\noindent
Similarly, 
let $\widehat{m}_{Y}$ and  
$\widehat{v}_{Y}$ be the sample average and variance
of the claim data $(y_{jk})_{j\in J,k=1,\ldots,n_j}$. 
Then we fit a distribution
for $Y$ such that
\[
\EE[Y] = \widehat{m}_{Y}\;\;\text{and}\;\;
\Var(Y) = \widehat{v}_{Y}.
\]
Note that the individual claim data $(y_{jk})$ are not
observed, but that their sample average can be computed:
\[
\widehat{m}_{Y}=\frac{\sum_j s_j}{\sum_j n_j}.
\]
And for the sample variance of the individual claims we 
use the well-known identity for
the variance of the aggregated sum $S=\sum_{k=1}^N Y_k$:
\[
\Var(S) = (\EE[N])(\Var(Y)) + (\Var(N))(\EE[Y])^2.
\]

\subsection{Subcategories Larger Cities}
630 data have insured customers from major cities. 

\begin{center}
\begin{tabular}{l|rr}
data  & average & variance \\
\hline
claim number & 70.60 & 52181.52\\
aggregate claim size & 329.22 & 1153532.32\\
individual claim size & 4.66 & 265.34
\end{tabular}
\end{center}

\noindent
These gave dispersion parameter $p$ of the counting distributions:

\begin{center}
\begin{tabular}{l|c}
& $p$\\
\hline
Abel &    2.695844\\
Arcsine & 2.598444 \\
Takacs &  3.821015 
\end{tabular}
\end{center}

\noindent
The parameters of the claim distributions were obtained as explained
above:

\begin{center}
\begin{tabular}{l|ccc}
& $\theta$ & $p$ & $\alpha$\\
\hline
Gamma  & 0.982425  & 0.081960 &\\
IG & -0.008788 & 2.616360 &\\
Stable & -0.015496 & 2.134192 & 0.118315
\end{tabular}
\end{center}

\noindent
With these parameters we have fitted the counting distribution 
and the claim distribution. Then we ran simulations
of aggregated claim sizes in these models and executed
the chi-square test for 
goodness-of-fit (hypothesising that the samples came from the same
distribution).
As an example, below we show the histograms of 
the data $S_N$ and the simulated $S_N$ in case of the
Poisson-Gamma, Abel-IG and Arcsine-Stable combinations.

\begin{figure}[H]
\centering
\includegraphics[width=0.6\textwidth]{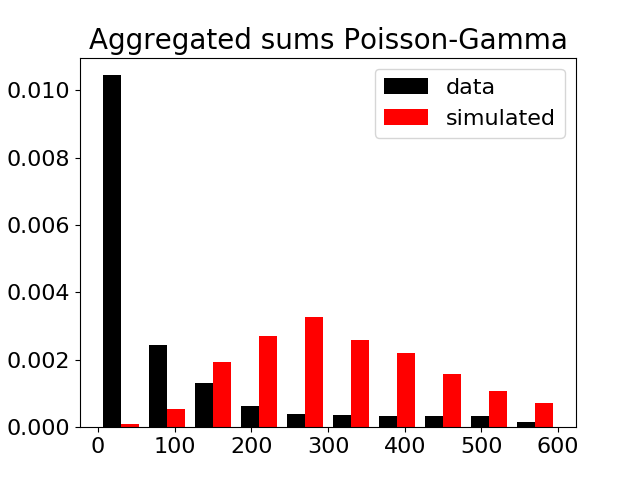}

\includegraphics[width=0.6\textwidth]{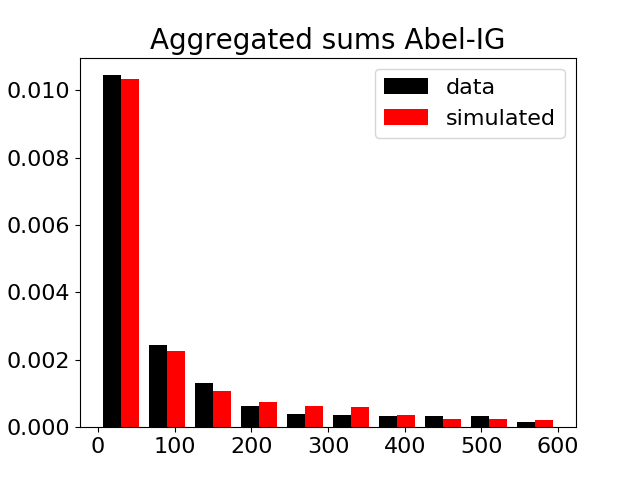}

\includegraphics[width=0.6\textwidth]{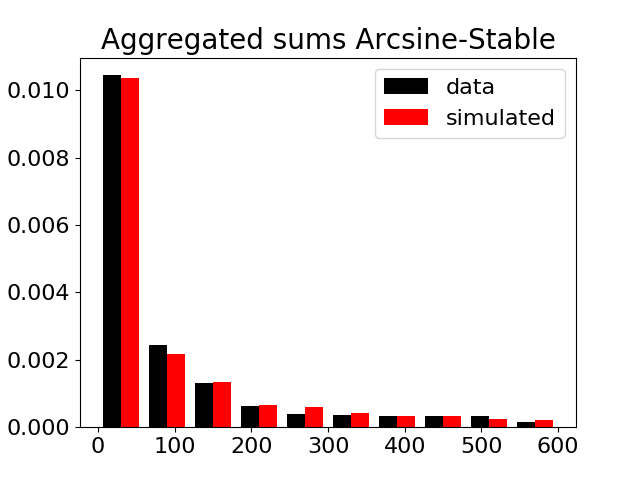}
\caption{
Histograms of the data and $2000$
simulated samples (normed to form pdf's).}%
\label{f:sumhist}%
\end{figure}

\noindent
Table \ref{t:pvalues} summarizes the test results in terms of $p$-values.

\begin{table}[H]
\caption{$p$-values of the fitted models.}
\label{t:pvalues}
\medskip
\centering
\begin{tabular}{ll|c}
count & claim & $p$-value\\
\hline
Poisson & Gamma & $\approx 0$\\ 
Poisson & IG & $\approx 0$\\ 
Poisson & Stable & $\approx 0$\\ 
\hline
Abel &  Gamma & $0.2101$\\
Abel &  IG & $0.2459$\\
Abel &  Stable & $0.3089$\\
\hline
Arcsine &  Gamma & $0.1306$\\
Arcsine &  IG & $0.4224$\\
Arcsine &  Stable & $0.7460$\\
\hline
Takacs &  Gamma & $0.4159$\\
Takacs &  IG & $0.2800$\\
Takacs &  Stable & $0.2701$\\
\end{tabular}
\end{table}

\noindent
We may conclude that the arcsine counting variable with
positive stable claim size gives the best fit.
The computations of the risk probabilities in this model
are easily implemented by the algorithms that we exposed
in Section \ref{s:counting} for the arcsine samples,
in Section \ref{s:claim} for the positive stable samples,
and  in Section \ref{s:computing} for the 
Monte Carlo and importance sampling simulations.
Table \ref{t:arcposstable} shows results for both
standard Monte Carlo and importance sampling simulations.
Again we see the exponential versus polynomial increase
of the required sample sizes.
For levels $x\leq 25000$ the estimates fall
in their corresponding confidence intervals
(in most caes), while for large levels, $x>25000$,
we have no Monte Carlo results.
 
\medskip

\begin{table}[H]
\caption{Estimates of $\PP(S_N>x)$ by Monte Carlo and imporance
sampling simulation
for arcsine counting and positive $\alpha$-stable claim
distributions.}
\label{t:arcposstable}
\medskip
\centering
\begin{tabular}{r||r|cc||r|cc}
& \multicolumn{3}{c||}{Monte Carlo}
& \multicolumn{3}{c||}{importance sampling}\\
$x$ & $M$ & $\widehat{\ell}$ & std. error & $M$ & $\widehat{\ell}$ & std. error\\
\hline
5000 & 10000 & 1.02e-02 & 1.00e-03 & 4000 & 9.89e-03 & 9.39e-04\\
10000 & 46000 & 2.11e-03  & 2.14e-04 & 7000 & 2.23e-03 & 2.23e-04\\
15000 & 128000 & 7.64e-04 & 7.73e-05  & 9000 & 8.12e-04 & 8.20e-05\\
20000 & 394000 & 2.49e-04 & 2.51e-05 & 14000 & 2.21e-04 & 2.25e-05\\
25000 & 1360000 & 7.13e-05 & 7.24e-06 & 16000 & 8.37e-05 & 8.52e-06\\
30000 & & & & 20000 & 4.18e-05 & 4.13e-06\\ 
35000 & & & & 24000 & 1.42e-05 & 1.42e-06 \\
40000 & & & & 30000 & 5.10e-06 & 5.15e-07\\
45000 & & & & 36000 & 2.31e-06 & 2.35e-07\\
50000 & & & & 38000 & 1.08e-06 & 1.08e-07
\end{tabular}
\end{table}

\section{Conclusion}
We analysed insurance claim data and modeled the accumulated 
claim during a certain a period as a random sum of positive
random variables representing the individual claims. 
The data showed that both the random sum and the random claim size
have variances as large as cubic powers of their means.
For fitting distributions with cubic variance functions to the
insurance data we used the NEF modeling. In this way we
considered three discrete counting variables 
for fitting the random sum, and three positive continuous distributions
for fitting the claim size, all coming from NEF's.
We gave a thorough analysis of the nontrivial 
discrete counting variables for the 
purpose of developing sampling algorithms.
These sampling algorithms are all accept-reject based, where
the dominating proposal distribution is a Zipf distribution.
Our claim size distributions are commonly known and 
sampling algorithms can be found in the literature.

Being able to sample from the aggregate claim distribution, 
we execute Monte Carlo simulations for computing tail
probabilities, specifically for large losses. The
efficiency of these simulations
was improved by two techniques. The first being that
the claim size distributions satisfy the reproducibility
property implying that convolutions come from the same family
as the individual distribution. The second improvement
is the application of importance sampling. Our numerical experiments
show that the exponential complexity of standard Monte Carlo is reduced
to polynomial complexity.

The downside of our method is that the accept-reject
algorithms of the counting distributions have acceptance
ratio of 25\%-35\%. Therefore we shall investigate in future
work the application of other sampling algorithms,
notably MCMC and multilevel Monte Carlo methods.

\bigskip\noindent
\textbf{Acknowledgement:}
This research is partially supported by the Netherlands Organisation
for Scientific Reserarch (NWO) project number 040.11.608.
The part of work of Ad Ridder is partially supported by the 
Zimmerman foundation while he was visiting the University of Haifa in 
January 2017.

\bibliographystyle{chicago}
\bibliography{references.bib}

\section*{Appendix}
\subsection*{Arcsine Bound}
Recall the kernel on the even outcomes

\[
\nu(2n)=
\frac{1}{(2n)!}\prod_{i=0}^{n-1}\big((2i)^2+p^2\big)
=\prod_{i=0}^{n-1}\frac{(2i)^2+p^2}{(2i+1)(2i+2)}
=\prod_{i=0}^{n-1}\underbrace{\frac{4i^2+p^2}{4i^2+6i+2}}_{=\rho_i}.
\]

\noindent
We shall bound the $\rho_i$'s for sufficiently large $i$. 
The threshold lies at $i^*$ that is such that both
$i^*+1i^*\geq 6$, and for all $i\geq i^*+1$,

\[
\frac{3}{2i} - \Big(\frac{9+p^2}{4i^2} - \frac{3p^2}{8i^3} + 
\frac{9p^2}{4i^4}\Big)>0. 
\] 

\noindent
First, dividing by $4i^2$ we easily get (for all $i\geq 1$),

\begin{align*}
&\rho_i=\frac{1 + p^2/(4i^2)}{1 + 6/(4i) + 2/(4i^2)}
\leq \frac{1 + p^2/(4i^2)}{1 + 6/(4i)}\\
&\leq \big(1 + \frac{p^2}{4i^2}\big)
\big(1 - \frac{6}{4i} + \frac{36}{(16i^2}\big)\\
&= 1 - \frac{3}{2i} + \Big(\underbrace{\frac{9+p^2}{4i^2} - \frac{3p^2}{8i^3} + 
\frac{9p^2}{4i^4}}_{=\epsilon_i}\Big).
\end{align*}

\noindent
Thus for $i\geq i^*+1$ we have that $\frac{3}{2i}-\epsilon_i>0$, and 
therefore we can bound

\[
\log \rho_i \leq \log\Big(1 - \big(\frac{3}{2i}-\epsilon_i\big)\Big)
\leq - \big(\frac{3}{2i}-\epsilon_i\big).
\] 

\noindent
Define $G=\prod_{i=0}^{i^*}\rho_i$. So we get for $n\geq i^*+1$,

\begin{align*}
\log &\,\nu(2n)=\log G + \sum_{i=i^*+1}^{n-1}\log \rho_i \\
&\leq \log G - \sum_{i=i^*+1}^{n-1} \big(\frac{3}{2i}-\epsilon_i\big)
= \log G - \frac32\sum_{i=i^*+1}^{n-1}\frac{1}{i}+
\sum_{i=i^*+1}^{n-1}\epsilon_i.
\end{align*}

\noindent
The second term is easy:

\[
\sum_{i=i^*+1}^{n-1}\frac{1}{i}
=\sum_{i=1}^{n-1}\frac{1}{i}-\sum_{i=1}^{i^*}\frac{1}{i}
\geq \log n -\sum_{i=1}^{i^*}\frac{1}{i}.
\]

\noindent
Concerning the $\epsilon_i$'s. Clearly positive, and for $i\geq 6$,

\begin{align*}
& \epsilon_i = \frac{9+p^2}{4i^2} - \frac{3p^2}{8i^3} + \frac{9p^2}{4i^4}\\
&=\frac{9}{4i^2} + \frac{p^2}{8i^2}
\Big(\underbrace{2- \frac{3}{i} + \frac{18}{i^2}}_{\leq 2}\Big) 
\leq \frac{9+p^2}{4i^2}.
\end{align*}

\noindent
Hence, for $n\geq i^*+1$,

\begin{align*}
&\sum_{i=i^*+1}^{n-1}\epsilon_i
\leq \frac{9+p^2}{4}\sum_{i=i^*+1}^{n-1} \frac{1}{i^2}
=\frac{9+p^2}{4}\Big(\sum_{i=1}^{n-1} \frac{1}{i^2}
-\sum_{i=1}^{i^*} \frac{1}{i^2}\Big) \\
&\leq \frac{9+p^2}{4}\Big(\sum_{i=1}^{\infty} \frac{1}{i^2}
-\sum_{i=1}^{i^*} \frac{1}{i^2}\Big) 
=\frac{9+p^2}{4}\Big(\zeta(2)-\sum_{i=1}^{i^*} \frac{1}{i^2}\Big),
\end{align*}

\noindent
where $\zeta(\cdot)$ is the Riemann-zeta function.
Wrapping up we get by exponentiating,
$\nu(2n) \leq K_1 n^{-3/2}$ for $n\geq i^*+1$,
where

\[
K_1 = G \,\exp\Big(\frac32\sum_{i=1}^{i^*}\frac{1}{i}
+ \frac{9+p^2}{4}\big(\zeta(2)-\sum_{i=1}^{i^*}\frac{1}{i^2}\big)\Big).
\]

\noindent
Find $K_0$ such that both $\nu(2n) \leq K_0 n^{-3/2}$ 
for $n=1,\ldots,i^*$. We demand this inequality also for 
the odd terms; i.e., $\nu(2n+1) \leq K_0 n^{-3/2}$,
$n=1,\ldots,i^*$.
Then by setting $K=\max\{K_0,K_1\}$,
\[
\nu(2n) \leq K\frac{1}{n\sqrt{n}},\quad n\geq 1.
\]

\noindent
Again we recognize the Zipf distribution, which will be usefull
for an accept-reject sampling algorithm.

\subsection*{Takacs Bound}
Recall the adapted kernel
\begin{align*}
& \nu_0(n)=\nu(n) e^{\theta_0(m)n}
=\frac{p}{n+p}\,\frac{1}{n!}\,(n+p)(n+p+1)\cdots(n+p+n-1)\,e^{\theta_0(m)n}\\
&=\frac{p}{n+p}\,\frac{(2n+p-1)!}{n!(n+p-1)!}\,e^{\theta_0(m)n},
\quad n=9,1,\ldots.
\end{align*}

\noindent
The Stirling bounds of $n!$ are
\[
\sqrt{2\pi n}\,\big(n e^{-1}\big)^n < n! 
< \sqrt{2\pi n}\,\big(n e^{-1}\big)^n\, e^{1/12},\quad n=1,2,\ldots.
\]

\noindent
Applying these bounds, we get

\[
\nu_0(n)\leq 
\frac{p}{n+p}\,\frac{e^{1/12}}{\sqrt{2\pi}}\,
\sqrt{\frac{2n+p-1}{n(n+p-1)}}\,
\Big(\frac{2n+p-1}{n+p-1}\Big)^{p-1}\,
\Big(\frac{(2n+p-1)^2}{n(n+p-1)}\,e^{\theta_0(m)}\Big)^{n}
\]

\noindent
The factors of this expression are worked out below.

\[
\sqrt{\frac{2n+p-1}{n(n+p-1)}} = \frac{1}{\sqrt{n}}\,
\sqrt{\frac{2+\frac{p-1}{n}}{1+\frac{p-1}{n}}}
\leq \frac{1}{\sqrt{n}}\,\sqrt{2},
\]
because in our models $p>1$. Thus also,
\[
\Big(\frac{2n+p-1}{n+p-1}\Big)^{p-1}=
\Big(\frac{2+\frac{p-1}{n}}{1+\frac{p-1}{n}}\Big)^{p-1}
\leq 2^{p-1}.
\]

\noindent
Finally,
\begin{align*}
&\frac{(2n+p-1)^2}{n(n+p-1)}\,e^{\theta_0(m)}
=\frac{(2n+p-1)^2}{n(n+p-1)}\,\frac{m(p+m)}{(p+2m)^2}\\
&=\frac{4n^2+4np-4n+(p-1)^2)}{n^2+np-n}\,\frac{mp+m^2}{p^2+4mp+4m^2}\\
&=\frac{n^2+np-n+(p-1)^2/4}{n^2+np-n}\,\frac{mp+m^2}{p^2/4+mp+m^2}\\
&= \Big(1 + \frac{(p-1)^2/4}{n^2+n(p-1)}\Big)\,
\Big(1+\frac{p^2/4}{m^2+mp}\Big)^{-1}.
\end{align*}
This expression is less than 1 for all $n\geq m$.
Putting it all together,

\[
\nu_0(n)\leq  
\frac{p}{n+p}\,\frac{e^{1/12}}{\sqrt{2\pi}}\,
\frac{\sqrt{2}}{\sqrt{n}}\,2^{p-1},
\quad n\geq m.
\]

\noindent
Let 
\[
K_1=p\,\frac{e^{1/12}}{\sqrt{2\pi}}\,
\sqrt{2}\,2^{p-1},
\]
then $\nu_0(n)\leq K_1 n^{-3/2}$ for all $n\geq m$.
Find $K_0$ such that $\nu_0(n) \leq K_0 n^{-3/2}$ 
for $n=1,\ldots,m$. 
Then by setting $K=\max\{K_0,K_1\}$,
\[
\nu_0(n) \leq K\frac{1}{n\sqrt{n}},\quad n\geq 1.
\]

\end{document}